\documentclass[runningheads]{llncs}
\usepackage{amssymb}
\usepackage[mathscr]{euscript}
\usepackage{scalerel}
\usepackage{graphicx}
\usepackage{hyperref}
\usepackage{verbatim,color}
\usepackage{url}
\usepackage{array} 
\usepackage{amssymb,amstext,amsmath,mathabx}
\usepackage[utf8]{inputenc}
\usepackage{mathtools}

\usepackage{amssymb,amstext,amsmath}
\newcommand{\no}[1]{\widebar{#1}}
\def\F{\mathcal F}

\def\M{\mathcal M}

\def\pr{\mathbb{P}}
\def\prev{\mathbb{P}}

\def\V{\mathcal{V}}
\def\C{\mathscr{C}}

\usepackage[utf8]{inputenc}
\usepackage{hyperref}
\usepackage{verbatim}
\usepackage{mathtools}
\usepackage{color}
\usepackage[normalem]{ulem} 
\usepackage{amssymb,amstext,amsmath,wasysym}

\title{Iterated Conditionals and \\ Characterization of P-entailment}
\author{Niki Pfeifer\inst{1}\thanks{Both authors contributed equally to the article and are listed alphabetically.
}\,\thanks{Supported by the BMBF project  01UL1906X.} 
\and  Giuseppe Sanfilippo\inst{2}$^\star$\;\thanks{Member of the GNAMPA Research Group 
and partially supported by the INdAM–GNAMPA Project 2020 Grant U-UFMBAZ-2020-000819.}
}

\author{Angelo Gilio\inst{1}
\thanks{Both authors contributed equally to the article and are listed alphabetically.
}\;
\thanks{Retired}\and  Giuseppe Sanfilippo\inst{2}$^\star$}
\institute{Department SBAI, University of Rome ``La Sapienza'', Rome, Italy
	\\ \email{angelo.gilio@sbai.uniroma1.it}
	\and
	Department of Mathematics and Computer Science,
	University of Palermo, Italy
	\\ \email{giuseppe.sanfilippo@unipa.it}
}
\authorrunning{Angelo Gilio and Giuseppe Sanfilippo}
\begin{document}
	\maketitle
\begin{abstract}
In this paper we deepen, in the setting of coherence, some results obtained in recent papers on the notion of p-entailment of Adams and its relationship with conjoined and iterated conditionals. We recall that conjoined and iterated conditionals are suitably defined in the framework of conditional random quantities. Given a family $\mathcal{F}$ of $n$ conditional events $\{E_1|H_1,\ldots, E_n|H_n\}$ we denote by $\mathscr{C}(\mathcal{F})=(E_1|H_1)\wedge \cdots \wedge  (E_n|H_n)$ the conjunction of the conditional events in $\F$. We introduce the iterated conditional $\mathscr{C}(\mathcal{F}_2)|\mathscr{C}(\mathcal{F}_1)$, where $\mathcal{F}_1$ and $\mathcal{F}_2$ are two finite families of conditional events, by showing that the prevision of  $\mathscr{C}(\mathcal{F}_2)\wedge \mathscr{C}(\mathcal{F}_1)$ is the product of the prevision of  $\mathscr{C}(\mathcal{F}_2)|\mathscr{C}(\mathcal{F}_1)$ and the prevision of $\mathscr{C}(\mathcal{F}_1)$.
Likewise the well known equality $(A\wedge H)|H=A|H$, we show that
$
(\mathscr{C}(\mathcal{F}_2)\wedge \mathscr{C}(\mathcal{F}_1))|\mathscr{C}(\mathcal{F}_1)=
\mathscr{C}(\mathcal{F}_2)|\mathscr{C}(\mathcal{F}_1)$. Then, we consider the case $\mathcal{F}_1=\mathcal{F}_2=\mathcal{F}$ and we verify for the prevision $\mu$ of $\mathscr{C}(\F)|\mathscr{C}(\mathcal{F})$ that the unique coherent assessment is  $\mu=1$ and, as a consequence,  $\mathscr{C}(\mathcal{F})|\mathscr{C}(\mathcal{F})$ coincides with the constant 1. Finally, by assuming $\mathcal{F}$ p-consistent, we deepen some previous characterizations of p-entailment by showing that $\mathcal{F}$ p-entails a  conditional event $E_{n+1}|H_{n+1}$ if and only if the iterated conditional $(E_{n+1}|H_{n+1})\,|\,\mathscr{C}(\mathcal{F})$ is constant and equal to 1.  We illustrate this characterization by an example related with weak transitivity.
\end{abstract}
\section{Introduction}
The study of logical operations among conditional events has been considered in many papers (see, e.g., \cite{benferhat97,CoSV13,CoSV15,Douven19,FlGH20,GoNW91,Kauf09,McGe89,NgWa94}).  In a pioneering paper, written in 1935, de Finetti (\cite{defi36}) proposed  a three-valued logic for conditional events. 
Many often, conjunctions and disjunctions  have been defined as suitable conditional events (see, e.g., \cite{adams75,Cala87,Cala17,CiDu12,CiDu13,GoNW91}). 
However, in this way   classical probabilistic properties are lost.  For instance, the lower and upper probability bounds for the conjunction  are no more   the Fr\'echet-Hoeffding bounds (\cite{SUM2018S}).
A more general approach to conjunction  has been given in  \cite{Kauf09,McGe89} and, in the setting of coherence,  in
\cite{GiSa13c,GiSa13a,GiSa14,GiSa17,GiSa19}, where also the notion of iterated conditional has been studied. 
In these papers the notions of compound and iterated  conditionals
are defined as suitable  {\em conditional random quantities} with a finite number of  possible values in the interval $[0,1]$. The  main relevance of our approach  is theoretical: indeed,  all the basic probabilistic properties  are preserved (for a synthesis see \cite{GiSa21}). For instance, 
 De Morgan’s Laws are satisfied  (\cite{GiSa19}) and
 the Fr\'echet-Hoeffding bounds for the conjunction of conditional events still hold (\cite{GiSa21IJAR}). A suitable notion of  conditional constituent can be introduced, with properties analogous to the case of unconditional events; moreover,  a generalized inclusion-exclusion formula for the disjunction of conditional events is valid (\cite{GiSa20}).
 We also recall that the Lewis’ triviality results (\cite{lewis76}) are avoided because in our theory  the Import-Export Principle is not valid  (see \cite{GiSa14,SGOP20,SPOG18}). 
For some applications of compound and iterated conditionals see, e.g., \cite{GOPS16,GiPS20,SGOP20,SaPG17,SPOG18}.
More specifically, by exploiting iterated conditionals, the  probabilistic modus ponens has been generalized to  conditional events  (\cite{SaPG17}); one-premise and two-premise centering inferences, 
related to the notion of  centering used in Lewis’ logic  (\cite{lewis73}), 
has been examined in 
\cite{GOPS16,SPOG18}.
In \cite{SGOP20}  several (generalized) iterated conditionals have been considered, in order to  properly formalize different kinds of latent information; in particular,  
some intuitive probabilistic assessments
discussed in  \cite{douven11b} have been explained,  by making explicit some  background information.

An interesting aspect which could be possibly investigated concerns the relationship of our notions of compound and iterated  conditionals  with other topics of research, such as belief and  plausibility functions, data fusion,  inductive reasoning, and fuzzy logic
(\cite{CoPV17,coletti04,CoPV20UMI,CoVa18,DuFP20,dubo16,DuLL07,GMMP09,PeVa20,SKIR20}). For instance, 
by recalling  \cite{coletti04}, 
an application of our notion of conjunction  could be given by  interpreting  the membership function of the cartesian product of fuzzy sets as
the  prevision of  conjoined conditionals.

By exploiting conjunction  a  characterization  of the  probabilistic entailment of Adams (\cite{adams75}) for conditionals has been given in  (\cite{GiSa19}). Moreover, by exploiting  iterated conditionals, 
the p-entailment  of $E_3|H_3$  from a p-consistent family $\F=\{E_1|H_1,E_2|H_2\}$ has been characterized by the property that  the iterated conditional $(E_3|H_3)|((E_1|H_1)\wedge (E_2|H_2))$ is constant and coincides with~1 (\cite{GiPS20}).
In this paper, based on a general notion of iterated conditional, we extend this characterization of p-entailment by considering the case where $\F$ is a family of $n$ conditional events.
\\
The paper is organized as follows. After recalling in Section \ref{SEC:2} some preliminary notions and results, in Section \ref{SEC:3} we introduce  the iterated conditional  $\C(\F_2)|\C(\F_1)$, where $\C(\F_1)$ and $\C(\F_2)$ are the conjunctions of the conditional events in two finite families $\F_1$ and $\F_2$. We show that $(\C(F_2)\wedge  \C(\F_1))|\C(\F_1)=\C(\F_2)|\C(\F_1)$
and 
 $\prev[\mathscr{C}(\mathcal{F}_2)\wedge \mathscr{C}(\mathcal{F}_1)]= \prev[\mathscr{C}(\mathcal{F}_2)|\mathscr{C}(\mathcal{F}_1)] \prev[\mathscr{C}(\mathcal{F}_1)]$. Then, we prove that  $\C(\F)|\C(\F)$ is constant and coincides with 1.
 In Section \ref{SEC:3}, by assuming  $\F$ p-consistent, we characterize the p-entailment of $E_{n+1}|H_{n+1}$ from $\F$ by the property that the iterated conditional  $(E_{n+1}|H_{n+1})|\C(\F)$ is constant and coincides with 1.  We also illustrate this characterization by an example related with weak transitivity. 
\section{Preliminary Notions and Results}
\label{SEC:2}
An event $A$ is
a two-valued logical entity which is either  \emph{true}, or \emph{false}.
We use
 the same symbol to refer to an  event and its indicator.
We  denote by
$\Omega$ the sure event and by $\emptyset$ the impossible one.
We denote by $A\land B$ (resp., $A\vee B$), or simply by $AB$, the  conjunction (resp., disjunction) of $A$ and $B$. By $\no{A}$ we denote the negation of $A$. 
We simply write $A \subseteq B$ to denote
that $A$ logically implies $B$.
Given two events $A$ and $H$, with $H \neq \emptyset$, the conditional event $A|H$  is a three-valued logical entity which is \emph{true}, or
	\emph{false}, or \emph{void}, according to whether $AH$ is true, or
	$\no{A}H$ is true, or $\no{H}$ is true, respectively. The negation $\no{A|H}$ of $A|H$ is defined as $\no{A}|H$.
	
In the betting framework, to assess $P(A|H)=x$ amounts to say that, for every real number $s$,  you are willing to pay 
an amount $s\,x$ and to receive $s$, or 0, or $s\, x$, according
to whether $AH$ is true, or $\no{A}H$ is true, or $\no{H}$
is true (bet called off), respectively. Hence, for the random gain $G=sH(A-x)$, the possible values are $s(1-x)$, or $-s\,x$, or $0$, according
to whether $AH$ is true, or $\no{A}H$ is true, or $\no{H}$
is true, respectively. 
We denote by $X$ a \emph{random quantity}, that is  an 
uncertain real quantity,  which has a well determined but unknown value. 
We assume that  $X$ has a finite set of possible values. Given any event $H\neq \emptyset$, 
agreeing to the betting metaphor, if you  assess that the prevision of $``X$ {\em conditional on} $H$'' (or short:  $``X$ {\em given} $H$''), $\pr(X|H)$, is equal to $\mu$, this means that for any given  real number $s$ you are willing to pay an amount $s\mu$ and to receive  $sX$, or $s\mu$, according  to whether $H$ is true, or  false (bet  called off), respectively. In particular, when $X$ is (the indicator of) an event $A$, then $\prev(X|H)=P(A|H)$.
Given a conditional event $A|H$
with  $P(A|H) = x$,
the indicator of $A|H$, denoted by the same symbol, is
\begin{equation}\label{EQ:AgH}
A|H=
AH+x \no{H}=AH+x (1-H)=\left\{\begin{array}{ll}
1, &\mbox{if $AH$ is true,}\\
0, &\mbox{if $\no{A}H$ is true,}\\
x, &\mbox{if $\no{H}$ is true.}\\
\end{array}
\right.
\end{equation} 
Notice that, denoting by $\prev$ the prevision, it holds that  $\prev(AH + x\no{H})=xP(H)+xP(\no{H})=x$. 
The third  value of the random quantity  $A|H$  (subjectively) depends on the assessed probability  $P(A|H)=x$. 
 When $H\subseteq A$ (i.e., $AH=H$), it holds that $P(A|H)=1$; then,
for the indicator $A|H$ it holds that 
\begin{equation}\label{EQ:AgH=1}
A|H=AH+x\no{H}=H+\no{H}=1, \;\; (\mbox{when }
H\subseteq A). 
\end{equation}
Likewise, if $AH=\emptyset$, it holds that $P(A|H)=0$; then
\[
A|H=0+0\no{H}=0, \;\; (\mbox{when }
AH=\emptyset). \]
For the indicator of the negation  of $A|H$ it holds that  $\no{A}|H=1-A|H$.
Given a random quantity $X$ and an event $H \neq \emptyset$, 
with a prevision assessment $\prev(X|H) = \mu$, in our approach, likewise formula (\ref{EQ:AgH}), the conditional random quantity $X|H$ is defined as 
$X|H=XH+\mu\no{H}$.
Notice that  $\prev(XH + \mu\no{H})=\prev(XH)+\mu P(\no{H})=\mu P(H)+\mu P(\no{H})=\mu$.  For a discussion on this extended notion of a conditional random quantity and on the notion of coherence of a prevision assessment see, e.g., \cite{GiSa14,GiSa20,SGOP20}. In betting terms \emph{coherence}  means  that \emph{ in any finite combination of $n$ bets, it cannot happen that, after discarding the cases where the bet is called off, the values  of the random gain  are all
positive, or all negative} (Dutch Book). 
\begin{remark}\label{REM:CONST}
Given a  conditional random quantity $X|H$ and a prevision assessment $\prev(X|H)=\mu$, if conditionally on $H$ being true  $X$ is constant, say $X=c$, then by coherence $\mu=c$.
\end{remark}
\emph{Probabilistic consistency and probabilistic entailment.}
We recall the notions of p-consistency and p-entailment of Adams
	(\cite{adams75}) formulated for conditional events in
	the setting of coherence in 
	 \cite{gilio13}  (see also \cite{biazzo05,gilio02,GiSa13IJAR}). 
	 For a discussion on deduction from uncertain premises and p-validity, under coherence, see \cite{Cruz20}.
	\begin{definition}
		\label{PC}
		Let $\mathcal{F}_{n} = \{E_{i}|H_{i} \, , \; i=1,\ldots ,n\}$ be a
		family of $n$ conditional events. Then, $\mathcal{F}_{n}$ is \emph{p-consistent}
		if and only if the probability assessment $(p_{1},p_{2},\ldots ,
		p_{n})=(1,1,\ldots ,1)$ on $\mathcal{F}_{n}$ is coherent.
	\end{definition}
	\begin{definition}
		\label{PE}
		A p-consistent family $\mathcal{F}_{n} = \{E_{i}|H_{i} \, , \; i=1,\ldots ,n\}$
		 \emph{p-entails} a conditional event $E|H$  (denoted by $\mathcal{F}
		_{n} \; \Rightarrow_{p} \; E|H$) if and only if for any coherent probability
		assessment $(p_{1},\ldots ,
				p_{n},z)$ on $\mathcal{F}_{n} \cup \{E|H
		\}$ it holds that: if $p_{1}=\cdots =p_{n}=1$, then $z=1$.
	\end{definition}
The inference from $\mathcal{F}_{n}$ to $E|H$ is p-valid if and only if
$\mathcal{F}_{n}  \Rightarrow_{p}  E|H$ (\cite{adams75}).\\
\emph{Logical operations among conditional events.}
We recall below the notion of conjunction of two conditional events (\cite{GiSa14}).
\begin{definition}\label{CONJUNCTION} Given any pair of conditional events $E_1|H_1$ and $E_2|H_2$, with $P(E_1|H_1)=x_1$ and $P(E_2|H_2)=x_2$, their conjunction $(E_1|H_1) \wedge (E_2|H_2)$ is the conditional random quantity defined as
	\begin{equation}\label{EQ:CONJUNCTION}
	\begin{array}{lll}
	(E_1|H_1) \wedge (E_2|H_2)&=&(E_1H_1E_2H_2+x_{1}\no{H}_1E_2H_2+x_{2}\no{H}_2E_1H_1)|(H_1\vee H_2) =\\
	&=&\left\{\begin{array}{ll}
	1, &\mbox{if $E_1H_1E_2H_2$ is true,}\\
	0, &\mbox{if $\no{E}_1H_1\vee \no{E}_2H_2$ is true,}\\
	x_1, &\mbox{if $\no{H}_1E_2H_2$ is true,}\\
	x_2, &\mbox{if $\no{H}_2E_1H_1$ is true,}\\
	x_{12}, &\mbox{if $\no{H}_1\no{H}_2$ is true},
	\end{array}
	\right.
	\end{array}
	\end{equation}		
	where $x_{12}=\prev[(E_1|H_1) \wedge (E_2|H_2)]=\prev[(E_1H_1E_2H_2+x_{1}\no{H}_1E_2H_2+x_{2}\no{H}_2E_1H_1)|(H_1\vee H_2)]$.
\end{definition}
In betting terms, the prevision $x_{12}$ represents the amount you agree to pay, with the proviso that you will receive the quantity $E_1H_1E_2H_2+x_{1}\no{H}_1E_2H_2+x_{2}\no{H}_2E_1H_1$, or you will receive back the quantity $x_{12}$, according to whether $H_1\vee H_2$ is true, or  $\no{H}_1\no{H}_2$ is true.
Notice that, differently from conditional events which are three-valued objects, the conjunction $(E_1|H_1) \wedge (E_2|H_2)$
is no longer a three-valued object, but  a five-valued object with values in $[0,1]$. 
We recall below the  notion of conjunction  of $n$ conditional events.
\begin{definition}\label{DEF:CONGn}	
	Let  $n$ conditional events $E_1|H_1,\ldots,E_n|H_n$ be given.
	For each  non-empty strict subset $S$  of $\{1,\ldots,n\}$,  let $x_{S}$ be a prevision assessment on $\bigwedge_{i\in S} (E_i|H_i)$.
	Then, the conjunction  $(E_1|H_1) \wedge \cdots \wedge (E_n|H_n)$ is the conditional random quantity $\C_{1\cdots n}$ defined as
	\begin{equation}\label{EQ:CF}
	\begin{array}{lll}
	\C_{1\cdots n}=
	[\bigwedge_{i=1}^n E_iH_i+\sum_{\emptyset \neq S\subset \{1,2\ldots,n\}}x_{S}(\bigwedge_{i\in S} \no{H}_i)\wedge(\bigwedge_{i\notin S} E_i{H}_i)]|(\bigvee_{i=1}^n H_i)=
	\\
	=\left\{
	\begin{array}{llll}
	1, &\mbox{ if } \bigwedge_{i=1}^n E_iH_i\, \mbox{ is true,} \\
	0, &\mbox{ if } \bigvee_{i=1}^n \no{E}_iH_i\, \mbox{ is true}, \\
	x_{S}, &\mbox{ if } (\bigwedge_{i\in S} \no{H}_i)\wedge(\bigwedge_{i\notin S} E_i{H}_i)\, \mbox{ is true}, \; \emptyset \neq S\subset \{1,2\ldots,n\},\\
	x_{1\cdots n}, &\mbox{ if } \bigwedge_{i=1}^n \no{H}_i \mbox{ is true},
	\end{array}
	\right.
	\end{array}
	\end{equation}
	where 
	\begin{equation}\label{EQ:PREVCN}
	\begin{array}{ll}
	x_{1\cdots n}=x_{\{1,\ldots, n\}}=\prev(\C_{1\cdots n})=\\
	=\prev[(\bigwedge_{i=1}^n E_iH_i+\sum_{\emptyset \neq S\subset \{1,2\ldots,n\}}x_{S}(\bigwedge_{i\in S} \no{H}_i)\wedge(\bigwedge_{i\notin S} E_i{H}_i))|(\bigvee_{i=1}^n H_i)].
	\end{array}
	\end{equation}
\end{definition}
For  $n=1$ we obtain $\C_1=E_1|H_1$.  In  Definition \ref{DEF:CONGn}  each possible value $x_S$ of $\C_{1\cdots n}$,  $\emptyset\neq  S\subset \{1,\ldots,n\}$, is evaluated  when defining (in a previous step) the conjunction $\C_{S}=\bigwedge_{i\in S} (E_i|H_i)$. 
Then, after the conditional prevision $x_{1\cdots n}$ is evaluated, $\C_{1\cdots n}$ is completely specified. Of course, 
we require coherence for  the prevision assessment $(x_{S}, \emptyset\neq  S\subseteq \{1,\ldots,n\})$, so that $\C_{1\cdots n}\in[0,1]$.
In the framework of the betting scheme, $x_{1\cdots n}$ is the amount that you agree to pay with the proviso that you will receive:\\
- the amount $1$, if all conditional events are true;\\
- the amount  $0$, if at least one of the conditional events is false; \\
- the amount $x_S$ equal to the prevision of the conjunction of that conditional events which are void,  otherwise. In particular you receive back $x_{1\cdots n}$ when all  conditional events are void.\\
As we can see from (\ref{EQ:CF}), the conjunction $\C_{1\cdots n}$ is (in general) a $(2^n+1)$-valued object because the number of nonempty subsets $S$, and hence the number of possible values $x_S$,  is $2^n-1$. 
We recall a result which shows that the prevision of the  conjunction on $n$  conditional events satisfies the Fréchet-Hoeffding bounds (\cite[Theorem13]{GiSa19}).
\begin{theorem}\label{THM:TEOREMAAI13}
	Let  $n$ conditional events $E_1|H_1,\ldots,E_{n}|H_{n}$ be given, with $x_i=P(E_i|H_i)$, $i=1,\ldots, n$  and $x_{1\cdots n}=\prev(\C_{1 \cdots n })$. Then
\[
	\max\{x_1+\cdots+x_{n}-n+1,0\}
	\,\,\leq \,\, x_{1\cdots n} \,\,\leq\,\, \min\{x_1,\ldots,x_n\}.
	\]
\end{theorem}
In \cite[Theorem 10]{GiSa21IJAR} we have shown, under logical independence, the sharpness of the Fréchet-Hoeffding bounds.
\begin{remark}\label{REM:CONGCONG}
Given a finite family  $\F$ of  conditional events,  their conjunction is also  denoted by $\C(\F)$. We recall that in \cite{GiSa19}, given  two finite families of conditional events $\F_1$ and $\F_2$,  the object $\C(\F_1) \wedge  \C(\F_2)$ is 
 defined as $\C(\F_1\cup \F_2)$. Then, conjunction satisfies the commutativity and associativity properties (\cite[Propositions 1 and 2]{GiSa19}). Moreover,  the operation of conjunction satisfies the monotonicity property  (\cite[Theorem7]{GiSa19}), that is
$
\C_{1\cdots n+1}\leq \C_{1\cdots n}.$
Then,
 \begin{equation}\label{EQ:MONOTGEN}
\C(\mathcal{F}_1\cup \mathcal{F}_2)\leq \C(\mathcal{F}_1),\;\;\C(\mathcal{F}_1\cup \mathcal{F}_2)\leq \C(\mathcal{F}_2).     
 \end{equation}
\end{remark}
\emph{Iterated conditioning.}
We now recall the notion of iterated conditional given in \cite{GiSa13a}. 
Such notion has the structure 
\def\mcirc{\mathbin{\scalerel*{\bigcirc}{t}}}
 $ \Box | \mcirc= \Box\wedge \mcirc +\prev(\Box|\mcirc)\no{\mcirc}$, 
 where $\prev$ denotes the prevision,
which reduces to formula (\ref{EQ:AgH})  when $\Box=A$ and $\mcirc=H$.
\begin{definition}[Iterated conditioning]
	\label{DEF:ITER-COND} Given any pair of conditional events $E_1|H_1$ and $E_2|H_2$, with $E_1H_1\neq \emptyset$, the iterated
	conditional $(E_2|H_2)|(E_1|H_1)$ is defined as the conditional random quantity
	\begin{equation}\label{EQ:ITER-COND}
	(E_2|H_2)|(E_1|H_1) = (E_2|H_2) \wedge (E_1|H_1) + \mu \no{E}_1|H_1,
	\end{equation}
	where
	$\mu =\mathbb{P}[(E_2|H_2)|(E_1|H_1)]$.
\end{definition}
\begin{remark}\label{REM:A|H=0}
	Notice that we assumed that $E_1H_1\neq \emptyset$ to give a nontrivial meaning to the notion of iterated conditional. Indeed,   if  $E_1H_1$ were equal to $\emptyset$,   then  $E_1|H_1=(E_2|H_2) \wedge (E_1|H_1)=0$ and  $\no{E}_1|H_1 =1$, from which it  would follow  
	$(E_2|H_2)|(E_1|H_1)=(E_2|H_2)|0=(E_2|H_2) \wedge (E_1|H_1) + \mu \no{E}_1|H_1=\mu$; that is, $(E_2|H_2)|(E_1|H_1)$ 
	would coincide with the (indeterminate) value $\mu$. Similarly to the case of a conditional event $E|H$, which is of no interest when $H=\emptyset$, the iterated conditional $(E_2|H_2)|(E_1|H_1)$ is not considered  in our approach when $E_1H_1=\emptyset$.
\end{remark}
Definition \ref{DEF:ITER-COND} has been generalized in \cite{GiSa19} to the case where the antecedent is the conjunction of more than two conditional events.
 \begin{definition}\label{DEF:GENITER}
	Let be given  $n+1$ conditional events $E_1|H_1, \ldots, E_{n+1}|H_{n+1}$, with $(E_1|H_1) \wedge \cdots \wedge (E_n|H_n)\neq 0$. We denote by $(E_{n+1}|H_{n+1})|((E_1|H_1) \wedge \cdots \wedge (E_n|H_n))=(E_{n+1}|H_{n+1})|\C_{1\cdots n}$ the random quantity
	\[
	\begin{array}{ll}
		(E_1|H_1) \wedge \cdots \wedge (E_{n+1}|H_{n+1})  + \mu \, (1-(E_1|H_1) \wedge \cdots \wedge (E_n|H_n))= \\
		=\C_{1\cdots n+1}+\mu \, (1-\C_{1\cdots n}),
	\end{array}
	\]
	where $\mu = \prev[(E_{n+1}|H_{n+1})|\C_{1\cdots n}]$.
\end{definition}
We observe that, based on the betting metaphor,  the quantity $\mu$ is the amount to be paid in order to receive the amount $\C_{1\cdots n+1}+\mu \, (1-\C_{1\cdots n})$. Definition \ref{DEF:GENITER} generalizes the notion  of iterated conditional $(E_{2}|H_{2})|(E_1|H_1)$ given in 
previous papers (see, e.g., \cite{GiSa13c,GiSa13a,GiSa14}). 
We also observe that, defining 	$\prev(\C_{1\cdots n})=x_{1\cdots n}$ and 
$\prev(\C_{1\cdots n+1})=x_{1\cdots n+1}$, by the linearity of prevision it holds that $\mu=x_{1\cdots n+1}+\mu \, (1-x_{1\cdots n})$; then, $x_{1\cdots n+1}=\mu \, x_{1\cdots n}$, that is 
$\prev(\C_{1\cdots n+1})=\prev[(E_{n+1}|H_{n+1})|\C_{1\cdots n}]\prev(\C_{1\cdots n})$. \\
\emph{Characterization of p-consistency and p-entailment.}
We recall a characterization of p-consistency of a family $\F$ in terms of the coherence of the prevision assessment  $\prev[\C(\F)]=1$ (\cite[Theorem 17]{GiSa19}).
\begin{theorem}\label{THM:PCC}
	A  family of $n$ conditional events $\F=\{E_1|H_1,\ldots, E_n|H_n\}$ is p-consistent if and only if the prevision assessment  $\prev[\C(\F)]=1$ is coherent.
\end{theorem}
We recall a characterization of  p-entailment in terms of  suitable  conjunctions. (\cite[Theorem 18]{GiSa19}).
\begin{theorem}\label{THM:PENT}
	Let be given a p-consistent family of $n$ conditional events $\F=\{E_1|H_1,\ldots, E_n|H_n\}$ and a further conditional event $E_{n+1}|H_{n+1}$. Then, the following assertions are equivalent:\\
	(i) $\F$ p-entails $E_{n+1}|H_{n+1}$;\\
	(ii) the conjunction $(E_1|H_1)\wedge\cdots \wedge(E_n|H_n)\wedge (E_{n+1}|H_{n+1})$ coincides with
	the  conjunction $(E_1|H_1)\wedge\cdots \wedge(E_n|H_n)$;\\
	(iii) the inequality  $(E_1|H_1)\wedge\cdots \wedge(E_n|H_n)\leq \; E_{n+1}|H_{n+1}$ is satisfied.
\end{theorem}
We recall a  result where it is shown that the p-entailment of a conditional event $E_3|H_3$ from a p-consistent family $\F=\{E_1|H_1, E_2|H_2\}$ is equivalent to condition  $(E_3|H_3)|((E_1|H_1)\wedge(E_2|H_2))=1$  (\cite[Theorem 8]{GiPS20}).
\begin{theorem}\label{THM:MAIN}
	Let  three conditional events $E_1|H_1$, $E_2|H_2$, and $E_3|H_3$ be given, where $\{E_1|H_1, E_2|H_2\}$ is p-consistent.
	Then, $\{E_1|H_1, E_2|H_2\}$ p-entails $E_3|H_3$ if and only if
	$(E_3|H_3)|((E_1|H_1)\wedge(E_2|H_2)) = 1$.
\end{theorem}
Theorem \ref{THM:MAIN} will be generalized in Section \ref{SEC:4}.
\section{A General Notion of Iterated Conditional}\label{SEC:3}
Let a family  $\F=\{E_1|H_1,\ldots,E_{n}|H_{n}\}$  of $n$ conditional events be given. Moreover, let   $\M=(x_{ S}:\emptyset \neq S \subseteq\{1,\ldots,n\})$ be a coherent prevision assessment  on the family
	$\{\C_{ S}:\emptyset \neq S \subseteq\{1,\ldots,n\}\}$, where $\C_{ S}=\bigwedge_{i\in S}(E_i|H_i)$ and $x_S=\prev(\C_{ S})$. Denoting by  $\Lambda$ the set of possible values of $\C(\F)$, it holds that  $\Lambda\subseteq  \{1,0,x_{S}:\emptyset\neq S\subseteq\{1,\ldots,n\}\}$. We observe that if $x_{S}=0$ for some $S$, then from (\ref{EQ:MONOTGEN}), it holds that $x_{S'}=0$ for every $S'$ such that $S\subset S'$. The conjunction $\C(\F)$ is constant and coincides with 0 when $\Lambda=\{0\}$, in which case we write $\C(\F)=0$. This  happens when
	$E_1H_1\cdots E_nH_n=\emptyset$ and  for each $\emptyset\neq S\subseteq\{1,\ldots,n\}$ such that 
	$ (\bigwedge_{i\in S} \no{H}_i)\wedge(\bigwedge_{i\notin S} E_i{H}_i)\neq \emptyset $ it holds that $x_{S}=0$.
	For instance,  when $E_1H_1\cdots E_nH_n=\emptyset$ and $x_1=\cdots=x_n=0$, it holds that $x_S=0$ for every $S$; then, $\Lambda=\{0\}$ and $\C(\F)$ coincides with the constant 0. 
We give below a generalization of Definition \ref{DEF:GENITER}.
\begin{definition}\label{DEF:GENERITER}
	Let $\F_1$ and $\F_2$ be two finite families of conditional events, with $\C(\F_1)\neq  0$. We denote by $\C(\F_2)|\C(\F_1)$ the random quantity defined as 
	\[
\C(\F_2)|\C(\F_1) =	\C(\F_2)\wedge\C(\F_1) + \mu (1-\C(\F_1)) 
		=\C(\F_1 \cup \F_2)+\mu (1-\C(\F_1)),
\]
where $\mu = \prev[\C(\F_2)|\C(\F_1)]$.
\end{definition}
We observe that Definition \ref{DEF:GENERITER} reduces to Definition \ref{DEF:GENITER} when  $\F_1 = \{E_1|H_1, \ldots, E_n|H_n\}$ and $\F_2=\{E_{n+1}|H_{n+1}\}$.
We also remark that  by linearity of prevision it holds that $\mu=\prev[\C(\F_2)|\C(\F_1)]=\prev[\C(\F_2)\wedge \C(\F_1)]+\mu (1-\prev[\C(\F_1)])$, that is
\begin{equation}\label{EQ:PREVCOMP}
\prev[\C(\F_2)\wedge \C(\F_1)]=\prev[\C(\F_2)|\C(\F_1)]\prev[\C(\F_1)].
\end{equation} Formula (\ref{EQ:PREVCOMP}) generalizes the well known relation: $P(A H)=P(A|H)P(H)$ ({\em compound probability theorem}). In the following result we obtain an equivalent representation of $\C(\F_2)|\C(\F_1)$.
\begin{theorem}\label{THM:CANITER}
	Let $\F_1$ and $\F_2$ be two finite families of conditional events, with $\C(\F_1)\neq 0$. It holds that 
\begin{equation}\label{EQ:CANITER}
\C(\F_2)|\C(\F_1)=
\C(\F_1\cup \F_2)|\C(\F_1)=(\C(\F_2)\wedge \C(\F_1))|\C(\F_1).
\end{equation}
\end{theorem}
\begin{proof}
We set $\mu'=\prev[\C(\F_2)|\C(\F_1)]$ and $\mu''=\prev[\C(\F_1\cup \F_2)|\C(\F_1)]$. Then,   
\[
\C(\F_2)|\C(\F_1)=
\C(\F_1\cup \F_2)+\mu'\,(1-\C(\F_1))
\]
and 
\[\C(\F_1\cup \F_2)|\C(\F_1)=
\C(\F_1\cup \F_2\cup \F_1)+\mu''\,(1-\C(\F_1))=\C(\F_1\cup \F_2)+\mu''\,(1-\C(\F_1)).
\]
In order to prove (\ref{EQ:CANITER}) it is enough to verify that $\mu'=\mu''$.
We observe that 
\[
\C(\F_2)|\C(\F_1)-\C(\F_1\cup \F_2)|\C(\F_1)=(\mu'-\mu'')(1-\C(\F_1)),
\]
where $\mu'-\mu''=
\prev[\C(\F_2)|\C(\F_1)-\C(\F_1\cup \F_2)|\C(\F_1)].$
Moreover, by setting $\F_1=\{E_1|H_1,\ldots,E_{n}|H_{n}\}$, it holds that 
\[
(\mu'-\mu'')(1-\C(\F_1))=
\left\{
\begin{array}{ll}
0, \mbox{ if } \C(\F_1)=1,\\
\mu'-\mu'', \mbox{ if } \C(\F_1)=0,\\
(\mu'-\mu'')(1-x_{S}), \mbox{ if } 0<\C(\F_1)=x_{S}<1,  
\end{array}
\right.
\]
where $\emptyset\neq S\subseteq\{1,\ldots,n\}$.
Within the betting framework, $\mu'-\mu''$ is the amount to be paid in order to receive the random amount $(\mu'-\mu'')(1-\C(\F_1))$. Then, as a necessary condition of coherence, $\mu'-\mu''$ must be a linear convex combination of the possible values of $(\mu'-\mu'')(1-\C(\F_1))$ associated with the cases where
the bet is not called off, that is the cases where you do not receive back the paid amount 
$\mu'-\mu''$. In other words, (as a necessary condition of coherence) $\mu'-\mu''$ must belong to the convex hull of the set $\{0,(\mu'-\mu'')(1-x_S):  0<x_{S}<1,\; \emptyset\neq S\subseteq\{1,\ldots,n\} \}$.
We observe that   $      \max\{0,|\mu'-\mu''|(1-x_S)\}\leq|\mu'-\mu''|$, where as $x_S\in(0,1)$ the equality holds if and only if $\mu'-\mu''=0$. Then,  $\mu'-\mu''$ belongs to convex hull of the set $\{0,(\mu'-\mu'')(1-x_S):  0<x_{S}<1,\; \emptyset\neq S\subseteq\{1,\ldots,n\} \}$ if and only if $\mu'-\mu''=0$, that is $\mu'=\mu''$. Thus,  $\C(\F_2)|\C(\F_1)=\C(\F_1\cup \F_2)|\C(\F_1)$.
Finally, by recalling  Remark \ref{REM:CONGCONG}, as $\C(\F_1\cup\F_2)=\C(\F_1)\wedge\C(\F_2)$, it follows that 
$\C(\F_2)|\C(\F_1)=(\C(\F_1)\wedge\C(\F_2))|\C(\F_1)$.
\qed
\end{proof}
In particular, given any family $\F=\{E_1|H_1,\ldots,E_n|H_n\}$, with $\C(\F)\neq 0$, and any conditional event $E|H$,   from  (\ref{EQ:CANITER}) it follows that 
\begin{equation}\label{EQ:CANITER1}
\C(\F\cup \{E|H\})|\C(\F)=(E|H)|\C(\F).
\end{equation}
 Moreover,  as $\F\cup \F=\F$, it holds that 
 \begin{equation}\label{EQ:CDATOC}
 \C(\F)|\C(\F)=\C(\F)+\mu(1-\C(\F)),
\end{equation}
 where $\mu=\prev[\C(\F)|\C(\F)]$. In the next theorem we show that $\C(\F)|\C(\F)=\mu=1$.
\begin{theorem}\label{THM:CDATOC}
Let $\F=\{E_1|H_1,\ldots, E_n|H_n\}$ be a  family of $n$ conditional events, with $\C(\F)$ not equal to the constant $0$. Then, $\C(\F)|\C(\F)$ coincides with the constant 1.
\end{theorem}  
\begin{proof}
We set $C_0=\no{H}_1\cdots \no{H}_n$. We observe that when $C_0$ is true, the value of $\C(\F)$ is $x_{1\cdots n}$ and the value of $\C(\F)|\C(\F)$ is $x_{1\cdots n}+\mu(1-x_{1\cdots n})$.
By linearity of prevision, from (\ref{EQ:CDATOC}) it holds that $\mu=x_{1\cdots n}+\mu(1-x_{1\cdots n})$.

We denote by $K$ the set of constituents $C_h$'s generated by $\F$ such that $C_h\subseteq H_1\vee \cdots \vee H_n$, that is $C_{h}\neq C_0$. Then, we consider the partition $\{K_1,K_0,K^*\}$ of $K$ as defined below. \\
\[
\begin{array}{ll}
K_1=\{C_h: \mbox{ if } C_h \mbox{ is true, then the value of } \C(\F) \mbox{ is 1}\},
\\
K_0=\{C_h: \mbox{ if } C_h \mbox{ is true, then the value of } \C(\F) \mbox{ is 0}\},\\
K^*=\{C_h: \mbox{ if } C_h \mbox{ is true, then the value of } \C(\F) \mbox{ is positive and less than 1}\}.
\end{array}
\]
Notice that the set $K_1$  also includes the constituents $C_h$'s such that $\C(\F)=x_S$, with $x_S=1$.  The set
 $K_0$ also includes the constituents $C_h$'s such that $\C(\F)=x_S$, with $x_S=0$.  
 For each $C_h\in K^*$, it holds that   
 \[
 C_h=(\bigwedge_{i\in S} \no{H}_i)\wedge(\bigwedge_{i\notin S} E_i{H}_i), \mbox{ for a suitable }\emptyset\neq S\subseteq\{1,\ldots, n\};
 \]
 moreover, if $C_h$ is true, then  $\C(\F)=x_S$, with $0<x_S<1$. 
 We observe that $K_1\cup K^*\neq \emptyset$,  because we assumed that $\C(\F)$ does not coincide with the constant 0.
 Moreover, the value of $\C(\F)|\C(\F)$ associated with a constituent $C_h$ is 1, or $\mu$, or belongs to the set $\{x_{S}+\mu(1-x_S): 0<x_{S}<1, \emptyset\neq S\subseteq\{1,\ldots,n\} \}$, according to whether $C_h\in K_1$, or $C_h\in K_0$, or $C_h\in K^*$, respectively.

By linearity of prevision, based on  (\ref{EQ:CDATOC}), it holds that
\[
\mu = \pr[\C(\F)|\C(\F)] = \pr[\C(\F)] + \mu \pr[(1 - \C(\F)] = x_{1 \cdots n} + \mu (1-x_{1 \cdots n}) \,,
\]
from which it follows that $\mu \, x_{1 \cdots n} = x_{1 \cdots n}$.\\ 
We distinguish two cases: $(a)$ $x_{1 \cdots n}>0$; $(b)$  $x_{1 \cdots n}=0$.  \\
$(a)$. As $x_{1 \cdots n}>0$, it holds that $\mu = 1$ and hence $x_S + \mu(1-x_S)=1$, for every $S$; therefore $\C(\F)|\C(\F)$ coincides with the constant 1. \\
$(b)$. In this case $x_{1 \cdots n}=0$. If we bet on $\C(\F)|\C(\F)$, we agree to pay its prevision $\mu$ by receiving  $\C(\F)|\C(\F)=\C(\F)+\mu (1-\C(\F))$, with the bet called off when  you receive back the paid amount $\mu$  (whatever be $\mu$). This happens when  it is true a constituent $C_h\in K_0\cup\{C_0\}$, in which case the value of $\C(\F)$ is 0, so that $\C(\F)+\mu (1-\C(\F))=\mu$. 
Denoting by $\Gamma$ the set of 
possible values of $\mathscr{C}(\mathcal{F})|\mathscr{C}(\mathcal{F})$, it holds that 
\[
\C(\F)|\C(\F) \in \Gamma \subseteq \mathcal{V}=\{1,\mu, x_S+\mu(1-x_S):  0<x_S<1, \emptyset\neq S\subseteq\{1,\ldots,n\} \}.
\]
Then, as a necessary condition of coherence, $\mu$ must belong to the convex hull of the set  
\[
\Gamma^*=\Gamma\setminus \{\mu\}\subseteq \V^*=\V\setminus \{\mu\}=\{1,x_S+\mu(1-x_S):  0<x_S<1, \emptyset\neq S\subseteq\{1,\ldots,n\} \}.
\]
Notice that $\Gamma^*$ is the set of values of $\C(\F)|\C(\F)$ associated with the constituents $C_h$'s which belong to the nonempty set $K_1\cup K^*$. Moreover, if $\mu$ does not belong to the convex hull of $\V^*$, then $\mu$ does not belong to the convex hull of $\Gamma^*$. In order $\mu$ be coherent it must belong to the convex hull of $\Gamma^*$, that is it must be a linear convex combination of the set of values in $\Gamma^*$. We distinguish three cases: $(i) \; \mu=1$; $(ii) \; \mu<1$; $(iii) \; \mu>1$. In the case $(i)$ it holds that $x_S+\mu(1-x_S)=1$, for every $S$; then $\Gamma^*=\V^*=\{1\}$, hence the prevision assessment $\mu=1$ is trivially coherent and $\C(\F)|\C(\F)=1$. In the case $(ii)$ it holds that $\mu < \min \V^*$ and hence $\mu$ doesn't belong to the convex hull of $\V^*$; then $\mu$ doesn't belong to the convex hull of $\Gamma^*$, that is the prevision assessment $\mu<1$ is not coherent.
In the case $(iii)$ it holds that $\mu > \max \V^*$ and hence $\mu$ doesn't belong to the convex hull of $\V^*$; then, $\mu$ doesn't belong to the convex hull of $\Gamma^*$, that is the prevision assessment $\mu>1$ is not coherent. Therefore, the unique coherent prevision assessment on $\C(\F)|\C(\F)$ is $\mu=1$ and hence 
\[
\C(\F)|\C(\F)=\C(\F)+\mu(1-\C(\F))=\C(\F)+1-\C(\F)=1 \,.
\]
\qed
\end{proof}
In the next section we generalize   Theorem \ref{THM:MAIN}, by characterizing the p-validity of  the inference from a premise set $\F=\{E_1|H_1,\ldots, E_n|H_n\}$ to the conclusion 
 $E_{n+1}|H_{n+1}$.

\section{Characterization of P-entailment in Terms of  Iterated Conditionals}\label{SEC:4}
We recall that,
given a family $\F$ of $n$ conditional events $\{E_1|H_1,\ldots,E_n|H_n\}$, we also denote by $\C_{1 \cdots n}$ the conjunction $\C(\F)$.
Moreover, given a further conditional event $E_{n+1}|H_{n+1}$ we denote by 
 $\C_{1 \cdots n+1}$ the conjunction $\C(\F\cup\{E_{n+1}|H_{n+1}\})$.
 In other words,
\[
\C_{1 \cdots n} = (E_1|H_1) \wedge \cdots \wedge (E_n|H_n),\;\;\C_{1 \cdots n+1} = (E_1|H_1) \wedge \cdots \wedge (E_{n+1}|H_{n+1}).\]
We set
$\pr(\C_{1 \cdots n}) = x_{1 \cdots n}$ and  $\pr(\C_{1 \cdots n+1}) = x_{1 \cdots n+1}$. Let us consider the iterated conditional $(E_{n+1}|H_{n+1})|\C_{1 \cdots n}$, which by Definition \ref{DEF:GENERITER}, is  given by 
\[
(E_{n+1}|H_{n+1})|\C_{1 \cdots n} = \C_{1 \cdots n+1} + \mu (1 - \C_{1 \cdots n}),
\]
where $\mu = \pr[(E_{n+1}|H_{n+1})|\C_{1 \cdots n}]$. In the next result, by assuming $\F$ p-consistent, the p-entailment of $E_{n+1}|H_{n+1}$ from   $\F$ is characterized in terms of the  iterated conditional $(E_{n+1}|H_{n+1})|\C_{1 \cdots n}$.
\begin{theorem}\label{THM:PEITER}
A p-consistent family $\F$ p-entails $E_{n+1}|H_{n+1}$ if and only if the iterated conditional $(E_{n+1}|H_{n+1})|\C_{1 \cdots n}$ is equal to 1.
\end{theorem}
\begin{proof}
First of all we observe that, by Theorem \ref{THM:PCC}, as $\F$ is p-consistent, the assessment  $\prev(\C_{1 \cdots n})=1$ is coherent and hence $\C_{1 \cdots n}\neq 0$, so that the iterated conditional  $(E_{n+1}|H_{n+1})|\C_{1 \cdots n}$ makes sense. We consider the following assertions: \\
$(i) \;  \F \;\;\mbox{p-entails}\;\; E_{n+1}|H_{n+1}$; 
$(ii) \; \C_{1 \cdots n+1} = \C_{1 \cdots n}$;
$(iii)\;  (E_{n+1}|H_{n+1})|\C_{1 \cdots n} = 1$. \\
By Theorem \ref{THM:PENT}, the conditions $(i)$ and $(ii)$ are equivalent, thus $(i) \;\;\Longrightarrow\;\; (ii)$. Then, in order to prove the theorem it is enough to verify that
\[
(ii) \;\;\Longrightarrow\;\; (iii) \;\;\Longrightarrow\;\; (i) \,.
\]
$(ii) \;\;\Longrightarrow\;\; (iii)$. 
By Theorem  \ref{THM:CANITER} it holds that 
$
(E_{n+1}|H_{n+1})|\C_{1 \cdots n}=\C_{1 \cdots n+1}|\C_{1 \cdots n}$.
Moreover, as $\C_{1 \cdots n+1}=\C_{1 \cdots n}$, it holds that  $\C_{1 \cdots n+1}|\C_{1 \cdots n}=\C_{1 \cdots n}|\C_{1 \cdots n}$, which  by  Theorem \ref{THM:CDATOC} is constant and coincides with 1.  Thus, 
$(E_{n+1}|H_{n+1})|\C_{1 \cdots n} = \C_{1 \cdots n}|\C_{1 \cdots n}= 1$ ,
that is $(iii)$ is satisfied.\\
$(iii) \;\;\Longrightarrow\;\; (i)$. As the condition $(iii)$ is satisfied, that is $(E_{n+1}|H_{n+1})|\C_{1 \cdots n}=1$, the unique coherent prevision assessment $\mu$ on $(E_{n+1}|H_{n+1})|\C_{1 \cdots n}$ is
 \[
\mu = 1 = \pr[\C_{1 \cdots n+1} + \mu (1- \C_{1 \cdots n})] = x_{1 \cdots n+1} + 1 - x_{1 \cdots n} \,,
 \]
 from which it follows $x_{1 \cdots n+1} = x_{1 \cdots n}$. We observe that, as $\F$ is p-consistent, it is coherent to assess $x_1= \cdots = x_n =1$. Moreover, by recalling Theorem \ref{THM:TEOREMAAI13}, it holds  that: 
 $
 \max \{x_1 + \cdots + x_n - n + 1, 0\} \;\leq\; x_{1 \cdots n} \;\leq\; \min \{x_1, \ldots, x_n\}$.\\
 Then, when $x_1= \cdots = x_n =1$
 it follows that $x_{1 \cdots n} = 1 = x_{1 \cdots n+1}$ and hence $x_{n+1}=1$. Thus, $\F$ p-entails $E_{n+1}|H_{n+1}$, that is the condition  $(i)$ is satisfied.
 \qed
\end{proof}
We recall that the Transitivity rule is not p-valid, that is $\{C|B,B|A\}$ does not p-entail $C|A$. In \cite[Theorem 5]{gilio16} it has been shown that
\[
P(C|B)=1,P(B|A)=1,P(A|(A\vee B))>0 \Rightarrow P ( C | A ) = 1,
\]
which is a weaker version of transitivity. This kind of Weak Transitivity has been also  obtained in \cite{freund1991} in the setting of preferential relations. In the next example, in order to  illustrate Theorems \ref{THM:CANITER}, \ref{THM:CDATOC} and \ref{THM:PEITER}, we consider two aspects: $(i)$ a p-valid version  of Weak Transitivity, where the constraint $P(A|(A\vee B))>0$ is replaced by $P(A|(A\vee B))=1$, by showing that $(C|A)|((C|B)\wedge(B|A)\wedge (A|(A\vee B)))=1$; $(ii)$  the non p-validity of the Transitivity rule, by showing that the iterated conditional $(C|A)|((C|B)\wedge(B|A))$ does not coincide with the constant 1. 
\begin{example}($i$)-\emph{Weak Transitivity}.
We consider the premise set $\F=\{C|B,B|A, A|(A\vee B)\}$ and the conclusion 
$C|A$, where the events $A,B,C$ are logically independent. 
By Definition \ref{DEF:CONGn}  
\[
\begin{array}{l}
\C(\F)=(C|B)\wedge(B|A)\wedge (A|(A\vee B))
=\left\{\begin{array}{ll}
	1, &\mbox{if $ABC$ is true,}\\	
	0, &\mbox{if $AB\no{C} \vee A\no{B}\vee \no{A}B$ is true,}\\
	z, &\mbox{if $\no{A}\no{B}$ is true,}\\	
\end{array}
\right.=\\
=ABC|(A\vee B),
\end{array}
\]
where $z=\prev[C(\F)]=P(ABC|(A\vee B))$. As $\C(\F)=ABC|(A\vee B)\subseteq C|A$, it follows that $\C(\F)\wedge (C|A)=\C(\F)$ and by  Theorem \ref{THM:PENT}, $\{C|B,B|A, A|(A\vee B)\}$ p-entails $C|A$.
Then,  by Theorem \ref{THM:PEITER},  
$(C|A)|((C|B)\wedge(B|A)\wedge (A|(A\vee B)))$ is constant and coincides with 1. Indeed, 
  by recalling Theorems \ref{THM:CANITER} and \ref{THM:CDATOC}, it holds that
\[
\begin{array}{ll}
(C|A)|((C|B)\wedge(B|A)\wedge (A|(A\vee B)))=(C|A)|\C(\F)=\\=((C|A)\wedge \C(\F))|\C(\F)=\C(\F)|\C(\F)=1.
\end{array}
\]
($ii$)-\emph{Transitivity}. We recall that Transitivity is not p-valid, that is the premise set $\{C|B,B|A\}$ does not p-entail the conclusion $C|A$. Then the iterated conditional $(C|A)|((C|B)\wedge(B|A))$ does not coincide with 1, as we show below. 
We set $P(B|A)=x$, $P(BC|A)=y$,   $\prev[(C|B)\wedge(B|A)]=u,$ $\prev[(C|B)\wedge(B|A)\wedge (C|A)]=w$, then by Definition \ref{CONJUNCTION} we obtain that \begin{equation*}
(C|B)\wedge(B|A)=(ABC+x\no{A}BC)|(A\vee B)=ABC+x\no{A}BC+u\no{A}\,\no{B},
\end{equation*}
and
\begin{equation*}\label{EQ:TR2}
\begin{array}{lll}
(C|B)\wedge(B|A)\wedge (C|A)= (C|B)\wedge(BC|A)=(ABC+y\no{A}BC)|(A\vee B)=\\=ABC+y\no{A}BC+w\no{A}\,\no{B}.
\end{array}
\end{equation*} 
Defining $\prev[(C|A)|((C|B)\wedge(B|A))]=\nu$, by linearity of prevision it holds that  $\nu=w+\nu(1-u)$ and  hence
\begin{equation*}\label{EQ:TR3}
\begin{array}{lll}
(C|A)|((C|B)\wedge(B|A))= (C|B)\wedge(B|A)\wedge(C|A)+\nu(1-(C|B)\wedge(B|A))=\\
=\left\{\begin{array}{ll}
1, &\mbox{if $ABC$ is true,}\\	
\nu, &\mbox{if $A\no{B}\vee B\no{C}$ is true,}\\
y+\nu(1-x), &\mbox{if $\no{A}BC$ is true,}\\	
\nu, &\mbox{if $\no{A}\no{B}$ is true.}\\	
\end{array}
\right.
\end{array}
\end{equation*} 
We observe that in general  $y+\nu(1-x)\neq 1$, for instance when $(x,y)=(1,0)$ it holds that $y+\nu(1-x)=0$. Thus,  in agreement with Theorem \ref{THM:PEITER}, the iterated conditional $(C|A)|((C|B)\wedge(B|A))$ does not coincide with the constant 1. 
\end{example}
Notice that  the p-validity of other inference rules, with a two-premise set $\{E_1|H_1,E_2|H_2\}$ and a conclusion $E_3|H_3$,  has been  examined in \cite{GiPS20} by checking whether the condition $(E_3|H_3)|((E_1|H_1)\wedge(E_2|H_2))=1$ is satisfied.
\section{Conclusions}\label{SEC:5}
In this paper, we generalized the notion of iterated conditional by introducing the random object $\C(\F_2)|\C(\F_1)$.  We showed that  $\prev[\mathscr{C}(\mathcal{F}_2)\wedge \mathscr{C}(\mathcal{F}_1)]= \prev[\mathscr{C}(\mathcal{F}_2)|\mathscr{C}(\mathcal{F}_1)] \prev[\mathscr{C}(\mathcal{F}_1)]$
and that 
$
(\mathscr{C}(\mathcal{F}_2)\wedge \mathscr{C}(\mathcal{F}_1))|\mathscr{C}(\mathcal{F}_1)=
\mathscr{C}(\mathcal{F}_2)|\mathscr{C}(\mathcal{F}_1)$. Then, we verified that the iterated conditional  $\C(\F)|\C(\F)$ is constant and coincides with 1. Moreover, under p-consistency of $\F$, we characterized the p-entailment of $E_{n+1}|H_{n+1}$ from $\F$ by the property that the iterated conditional where the antecedent is the conjunction $\C(\F)$ and the consequent is  $(E_{n+1}|H_{n+1})$ coincides with the constant 1. In other words, $\F$ p-entails $(E_{n+1}|H_{n+1})$ if and only if  $(E_{n+1}|H_{n+1})|\C(\F)=1$.  We have also illustrated this characterization by an example related with weak transitivity. We observe that a particular case of this characterization is obtained when we consider a (p-consistent) family of $n$ unconditional events $\F=\{E_1,\ldots, E_n\}$   and a further event $E_{n+1}$. In this case  $\F$ p-entails $E_{n+1}$ if and only if $E_1\cdots E_n\subseteq E_{n+1}$, which also amounts to the property that the conditional event  $E_{n+1}|E_1\cdots E_n$ coincides with 1. 
\\ \ \\
{\bf Acknowledgements.} We thank the four anonymous reviewers for their useful comments and suggestions. G. Sanfilippo has been partially supported by the INdAM–GNAMPA Project 2020 Grant U-UFMBAZ-2020-000819.

\bibliographystyle{splncs04} 
\bibliography{bibfile}

\end{document}